\documentclass[12pt]{article}
\usepackage{amssymb,latexsym,amscd,amsmath,amsfonts,amsthm,enumerate}

\newtheorem{theorem}{Theorem}[section]

\newtheorem{proposition}[theorem]{Proposition}
\newtheorem{corollary}[theorem]{Corollary}

\newtheoremstyle{definition}
  {6pt}
  {6pt}
  {}
  {}
  {\bfseries}
  {.}
  {.5em}
  {}%

\theoremstyle{definition}

\newtheoremstyle{remark}
  {6pt}
  {6pt}
  {}
  {}
  {\bfseries}
  {.}
  {.5em}
  {}%

\theoremstyle{remark}
\newtheorem{remark}[theorem]{Remark}
\newtheorem*{thr}{Main theorem}

\makeatletter
\renewcommand\@makefntext[1]{%
\setlength\parindent{1em}%
\noindent

\makebox[1.8em][r]{}{#1}}
\makeatother

\begin{document}
\title{\bf  \large THE STRUCTURE OF PERFECT FIELDS }
\author{
 \centerline{\small Duong Quoc Viet and Truong Thi Hong Thanh }\\}
   \date{}
\maketitle
\centerline{\parbox[c]{10.45 cm}{
\small {\bf ABSTRACT:}  This paper  builds   fundamental  perfect fields  of  positive  characteristic and shows  the structure of   perfect fields that  a field of  positive  characteristic is a  perfect field if and only if it is an algebraic extension of a fundamental perfect field. }}

\section{Introduction}
It has long been known that perfect fields are significant because Galois theory over these fields becomes simpler,  since the separable condition (the  most  important condition
of  Galois  extensions)  of  algebraic extensions over perfect fields are automatically satisfied.
Recall that  a  field $k$ is said to be a {\it  perfect  field}  if any one of the following equivalent conditions holds:
 \begin{enumerate}[\rm (a)]

 \item Every irreducible polynomial over $k$ is separable.
 \item  Every finite-degree   extension of $k$ is separable.
 \item  Every algebraic extension of $k$ is separable.
 \end{enumerate}

One showed  that  any field of  characteristic $0$ is a perfect field. And
 any field $k$  of  characteristic $p > 0$ is a perfect field if and only if  the Frobenius endomorphism $x \mapsto x^p$ is an automorphism of $k;$ i.e., every element of $k$ is a  $p$-th power (see e.g. \cite {BN,  M,  BO, S}).
In this paper, we   show  the structure of  perfect fields  of  positive  characteristic.

Let $p$ be a positive prime and    $X = \{x_i\}_{i \in I}$ a set of independent variables.
Denote by $\mathbb{Z}_p$  the field of congruence classes modulo $p$ and
$\mathbb{Z}_p(X)$  the field   of all rational functions in $X.$ For any  $u \in  \mathbb{Z}_p(X)$ and any $n \geq 0,$   we  define
$\sqrt[p^n]{u}$ to be the unique  root of $ y^{p^n} - u \in \mathbb{Z}_p(X)[y] $ in an algebraic closure $\overline{\mathbb{Z}_p(X)}$ of     $\mathbb{Z}_p(X).$
Set  $\sqrt[p^n]{X} = \{\sqrt[p^n]{x_i}\}_{i \in I}$ for all $n \ge 0.$
 It is clear that
 $\mathbb{Z}_p(\sqrt[p^n]{X}) \subseteq \mathbb{Z}_p(\sqrt[p^{n+1}]{X})$ for all $n \ge 0.$
Then we prove that   $F_p(I) = \bigcup_{n \ge 0}  \mathbb{Z}_p(\sqrt[p^n]{X})$ is  a perfect closure of  $\mathbb{Z}_p(X)$ (see  Proposition \ref{lm21}).
 Note that  we assign $F_p(I) = \mathbb{Z}_p$
   to  the case that $I = \emptyset$,   and  call a field $k$  of  characteristic $p > 0$ is a {\it fundamental perfect field} if $k \cong F_p(I)$ for some $I.$

\footnotetext{\begin{itemize}
\item[ ]{\it Mathematics  Subject
Classification}(2010): Primary  12F05.    Secondary 12F10; 12F15;
12F20. \item[ ]{\it Key words and phrases}: Perfect  fields,
Galois theory, fields of  positive  characteristic.
\end{itemize}}

 As one might expect,  we  obtain the following result for  the structure of  perfect fields  of  positive  characteristic.

\begin{thr}[see Theorem \ref{th2.2}] {\it  A field of  positive  characteristic is a  perfect field if and only if it is an algebraic extension of a fundamental perfect field}.
\end{thr}

The remainder of the paper (Section 2) is  devoted  to  the discussion of fundamental  perfect fields (see Proposition \ref{lm21}), and  proves  the main theorem.  Moreover, we  obtain other  interesting information on maximal fundamental  perfect  subfields of perfect fields (see  Corollary  \ref{co2.3} and Proposition \ref{lm2.4}).

\section{Fundamental perfect fields }
In this section, we   will  build  fundamental  perfect  fields  of  positive  characteristic and show  the structure of   perfect fields.

Recall that  a {\it perfect closure} $F$  of  a  field $E$  is  a  smallest  perfect  extension of $E$ in  an algebraic closure $\overline{E}$ of  $E;$  i.e., $F \subseteq F'$ for all  $F'$  perfect and  $ E \subseteq F'  \subseteq \overline{E}.$
 Let $p$ be a positive  prime.
Denote by  $\mathbb{Z}_p$  the field of congruence classes modulo $p.$
And throughout this paper,  we  always  consider  $\mathbb{Z}_p$ as a  subfield  of fields of   characteristic $p.$

\begin{remark} \label{lm20} Let  $H$ be a field  of   characteristic $p$
 and $a \in H.$  Then for any $n \ge 0, $  the  polynomial  $x^{p^n} - a$  has a   unique root in  an algebraic closure of $H$ and denote by
$\sqrt[p^n]{a}$ this unique root. And if   $H$  is  perfect  then
     $\sqrt[p^n]{a} \in H.$
\end{remark}

Let   $X = \{x_i\}_{i \in I}$  be a set of independent variables and  assume that  $\mathbb{Z}_p(X)$ is  the field   of all rational functions in $X$  over   $\mathbb{Z}_p.$
For any  $u \in \mathbb{Z}_p(X)$ and any $n \ge 0$,   denote  by
$\sqrt[p^n]{u}$ the  root of  $ y^{p^n} - u $ in an algebraic closure $\overline{\mathbb{Z}_p(X)}$ of     $\mathbb{Z}_p(X).$  Set  $\sqrt[p^n]{X} = \{\sqrt[p^n]{x_i}\}_{i \in I}$ for all $n \ge 0.$
It is clear that   $\sqrt[p^{n}]{X} = (\sqrt[p^{n+1}]{X})^p = \{  (\sqrt[p^{n+1}]{x_i})^p\}_{i \in I}$
 and
  $\mathbb{Z}_p(\sqrt[p^n]{X}) \subseteq \mathbb{Z}_p(\sqrt[p^{n+1}]{X})$ for all $n \ge 0.$ Consequently,
 $F_p(I) = \bigcup_{n \ge 0}  \mathbb{Z}_p(\sqrt[p^n]{X})$ is  a  subfield of $\overline{\mathbb{Z}_p(X)}.$
 Now,  we assign $F_p(I) = \mathbb{Z}_p$
   to  the case that $I = \emptyset.$
      Note that the structure of
$F_p(I)$  only depends  on the cardinality  of  $I;$ i.e., if  $ |I| = |J|$ then  $F_p(I) \cong F_p(J).$

   Then we have the following result.
\begin{proposition} \label{lm21} $F_p(I)$ is a perfect closure of  $\mathbb{Z}_p(X).$
\end{proposition}

\begin{proof} Let $a$ be an element of $F_p(I).$  There  exists $n$ such that $a \in \mathbb{Z}_p(\sqrt[p^n]{X}).$
 Hence $$a = \frac{g(\sqrt[p^n]{x_1},\ldots, \sqrt[p^n]{x_d} )}{h(\sqrt[p^n]{x_1},\ldots, \sqrt[p^n]{x_d} )}$$ for
$g(y_1, \ldots, y_d)$;  $h(y_1, \ldots, y_d)$ are two  polynomials of the  polynomial  ring  $ \mathbb{Z}_p[y_1, \ldots, y_d]$ in $y_1, \ldots, y_d$ over  $\mathbb{Z}_p.$  Since  $\mathbb{Z}_p$ is perfect,  it follows that  $$g(y_1^p, \ldots, y_d^p) = u(y_1, \ldots, y_d)^p \;\;\; \text{ and}\;\;\;  h(y_1^p, \ldots, y_d^p) = v(y_1, \ldots, y_d)^p$$ for $u(y_1, \ldots, y_d); v(y_1, \ldots, y_d) \in \mathbb{Z}_p[y_1, \ldots, y_d].$
Consequently,  we get $$a = \frac{g(\sqrt[p^n]{x_1},\ldots, \sqrt[p^n]{x_d} )}{h(\sqrt[p^n]{x_1},\ldots, \sqrt[p^n]{x_d} )}= \frac{g((\sqrt[p^{n+1}]{x_1})^p,\ldots, (\sqrt[p^{n+1}]{x_d})^p )}{h((\sqrt[p^{n+1}]{x_1})^p,\ldots, (\sqrt[p^{n+1}]{x_d})^p )} = \frac{u(\sqrt[p^{n+1}]{x_1},\ldots, \sqrt[p^{n+1}]{x_d} )^p}{v(\sqrt[p^{n+1}]{x_1},\ldots, \sqrt[p^{n+1}]{x_d} )^p}. $$
 Hence $a = b^p$ for $b = \frac{u(\sqrt[p^{n+1}]{x_1},\ldots, \sqrt[p^{n+1}]{x_d} )}{v(\sqrt[p^{n+1}]{x_1},\ldots, \sqrt[p^{n+1}]{x_d} )} \in \mathbb{Z}_p(\sqrt[p^{n+1}]{X}) \subseteq F_p(I).$
So $F_p(I)$ is perfect. Now, assume that $H$ is a perfect field and $ \mathbb{Z}_p(X) \subseteq H \subseteq \overline{\mathbb{Z}_p(X)}.$  Since $X \subseteq H$,  it follows   that  $ \sqrt[p^n]{X} \subseteq H$ for all $n \ge 0$ by  Remark \ref{lm20}.
Hence  $F_p(I) \subseteq H. $  Thus, $F_p(I)$ is a perfect closure of  $\mathbb{Z}_p(X).$
\end{proof}

We call a field $k$  of  characteristic $p > 0$ is a {\it fundamental perfect field} if $k \cong F_p(I)$ for some $I.$
Then we  obtain the  following interesting result.

\begin{theorem}\label{th2.2} A field of  positive  characteristic is a  perfect field if and only if it is an algebraic extension of a fundamental perfect field.
\end{theorem}
\begin{proof}
Let $E$ be a  perfect field of  positive  characteristic $p.$   Assume that
$S = \{s_i\}_{i \in I}$ is a transcendence    base  of $E$  over   $\mathbb{Z}_p.$ Then it is easily seen that $E$ is algebraic over $\mathbb{Z}_p(S).$  Let   $X = \{x_i\}_{i \in I}$  be a set of independent variables.  Since $S = \{s_i\}_{i \in I}$ is  algebraically independent   over   $\mathbb{Z}_p,$ it  follows  that $\mathbb{Z}_p(S) \cong \mathbb{Z}_p(X).$
Let  $F$ be a  perfect closure of $\mathbb{Z}_p(S).$  Since $E$ is  a  perfect field, we can consider $F \subseteq E.$
Remember that $F_p(I)$ is a perfect closure of  $\mathbb{Z}_p(X)$
by Proposition \ref{lm21}
and  $\mathbb{Z}_p(S) \cong \mathbb{Z}_p(X),$
$F \cong  F_p(I)$ by \cite {BN}.    So $F$ is a  fundamental perfect field. And since $E$ is algebraic over $\mathbb{Z}_p(S),$ $E$ is algebraic over $F.$  Conversely,  if  $K$  is an  algebraic extension of a  fundamental perfect field $k$  then $K$ is perfect. Indeed, for  any algebraic extension $H$ of $K,$ $H$ is an  algebraic extension of $k.$
 Since $k$ is perfect,  $H$ is  separable over $k.$  Hence  $H$ is  separable over $K.$  Thus, $K$ is perfect.
The proof is complete.
\end{proof}

 A  subfield $F$ of a perfect  field $P$ is called a  {\it maximal  fundamental  perfect  subfield} of $P$  if $F$ is  a  fundamental  perfect  field and  $P$ is an algebraic extension of  $F.$

 Then from the proof of Theorem  \ref {th2.2}, we find  the following corollary.

 \begin{corollary}\label{co2.3} Let $P$ be a perfect  field  of positive  characteristic $p.$ Then the following statements hold.
 \begin{enumerate}[\rm (i)]

 \item   $Q$ is  a  maximal  fundamental  perfect  subfield of $P$  if and only if there exists is a transcendence  base  $S = \{s_i\}_{i \in I}$  of $P$  over   $\mathbb{Z}_p$  such that   $Q$ is  the  perfect closure  of  $\mathbb{Z}_p(S)$ in $P.$ In this case,  $Q \cong F_p(I).$

 \item  The maximal  fundamental  perfect  subfields  of  $P$ are  isomorphic.

 \item  The structure of   maximal  fundamental  perfect  subfields of $P$   only depends  on the cardinality  of  transcendence  bases  of  $P$   over   $\mathbb{Z}_p.$

 \end{enumerate}
 \end{corollary}
\begin{proof} $Q$ is  a  maximal  fundamental  perfect  subfield of $P$  if and only if $Q \cong F_p(I)$ for some $I$  and $P$ is an algebraic extension of  $Q.$ On the one hand,  $Q \cong F_p(I)$
if and only if  there exists is a transcendence  base  $S = \{s_i\}_{i \in I}$  of $Q$  over   $\mathbb{Z}_p$  such that   $Q$ is  the  perfect closure  of  $\mathbb{Z}_p(S)$ in $P$  by the proof of  Theorem  \ref {th2.2}. On the other hand,
 $P$ is an algebraic extension of  $Q$ if and only if  $S = \{s_i\}_{i \in I}$ is  a transcendence base of $P$ over $\mathbb{Z}_p.$ So we get (i). Since  if  $ |I| = |J|$ then $F_p(I) \cong F_p(J),$
  hence from the proof of (i) and  remember  that  the cardinalities  of  transcendence  bases of  $P$    over   $\mathbb{Z}_p$ are the same, we have (ii).
         From (i) and  the proof of (ii) we get (iii).
\end{proof}

Finally, we  would like to  give the following proposition for perfect closures of fields.

\begin{proposition} \label{lm2.4}  Let $E$  be a  field  of positive  characteristic $p$ and
$S = \{s_i\}_{i \in I}$  a  transcendence  base  of  $E$  over   $\mathbb{Z}_p.$  Let $V = \{v_j\}_{j \in J}$  be  a  base for    the vector  space $E$ over $\mathbb{Z}_p(S)$ and $Q$ a  perfect closure of  $\mathbb{Z}_p(S)$ in an algebraic closure $\overline{E}$ of $E.$  Then $Q(V)$ is a  perfect closure of $E,$ and $E$ is perfect if and only if $Q\subseteq E.$
\end{proposition}
\begin{proof}  Since $V$ is  algebraic over $\mathbb{Z}_p(S),$  $Q(V)$ algebraic over $Q.$ Hence
 $Q(V)$ is perfect. Now, assume that $E \subseteq F \subseteq \overline{E}$ is perfect. Then $Q \subseteq F$ and hence
$Q(V) \subseteq F.$  So $Q(V)$ is a  perfect closure of $E.$ From this it follows that  $E$ is perfect if and only if
$E = Q(V).$  This  is  equivalent  to  $Q \subseteq E.$
\end{proof}

\noindent
Department of Mathematics\\
\noindent
Hanoi National University of Education\\
\noindent
136 Xuan Thuy street, Hanoi, Vietnam\\
\noindent
 duongquocviet@fmail.vnn.vn  and  thanhtth@hnue.edu.vn
\end{document}